\documentclass[12pt]{article}
\PassOptionsToPackage{table,xcdraw,dvipsnames}{xcolor}
\usepackage{xcolor}
\usepackage{amsmath,amssymb,amsthm}
\usepackage[T1]{fontenc}
\usepackage{a4wide}
\usepackage{lipsum}
\usepackage{graphicx}
\usepackage{enumitem}
\usepackage[linesnumbered,ruled,vlined]{algorithm2e}
\usepackage{tikz}
\usepackage[table,xcdraw,dvipsnames]{xcolor}
\usepackage{subcaption}
\usepackage{caption}
\usetikzlibrary{graphs}
\usepackage{amsmath,amsfonts,amsthm}
\usepackage{hyperref}
\usepackage{mathtools}
\usepackage{amssymb}
\usepackage[utf8]{inputenc}
\usepackage{float}
\usepackage{placeins} 
\usepackage[english]{babel}  

\newtheorem{defin}{Definition}

\newtheorem{utv}{Proposition}

\tikzset{every node/.style={circle, draw}}

\title{Algorithm for computing the partition function of the Potts model for SP-graphs}

\author{
Sofya Mukhamedzhanova\thanks{Institute of Mathematics and Mechanics, Kazan Federal University, Kremlevskaya ul. 18, Kazan, 420008 Russia.
E-mail: {\tt sofyaenglish@gmail.com}. Supported by RSF (project No. 24-21-00158).} \\
Bulat Sabirov\thanks{Institute of Information technology and Intelligent Systems, Kazan Federal University, Kremlevskaya ul. 18, Kazan, 420008 Russia.
E-mail: {\tt bulatsabirow@gmail.com}.} \\
Amir Mukhamedzhanov\thanks{Institute of Information technology and Intelligent Systems, Kazan Federal University, Kremlevskaya ul. 18, Kazan, 420008 Russia.
E-mail: {\tt atrukamir@gmail.com}.}
}

\makeatletter

\newcommand{\shorttitle}{\@title}

\def\@maketitle{%
  \newpage
  \begin{center}%
  \let \footnote \thanks
    {\small
    }
    \vskip 0.5em
    \rule{\linewidth}{0.04cm}
    \vskip 3.5em
    {\LARGE \textbf{\textsc{\@title}} \par}%
    \vskip 1.5em
    \vskip 2.5em%
    {\large
      \lineskip .5em%
      \begin{tabular}[t]{c}%
        \@author
      \end{tabular}\par}%
  \end{center}%
  \par
  }
\makeatother

\definecolor{lightgreen}{HTML}{098842}
\definecolor{lightblue}{HTML}{3886BC}
\definecolor{purple}{HTML}{55007D}

\begin{document}
\thispagestyle{empty}
\maketitle

\begin{abstract}
The q-state Potts model is a fundamental framework in statistical physics and graph theory, with its partition function encoding rich information about spin configurations. The multivariate Tutte polynomial (known as the partition function of the Potts model) can be defined on an arbitrary finite graph $G$  and encodes a lot of important combinatorial information about the graph. 
As a special case, it contains the familiar Tutte polynomial with two variables
and, consequently, its specialization with one variable, such as the chromatic polynomial, the flow polynomial and the reliability polynomial.
The main goal of this paper is to present an efficient algorithm for computing the Potts model partition function on SP-graphs (series-parallel graphs) with arbitrary weights. The algorithm for SP-graphs is based on simplifying the graph by replacing several edges with a single edge of equivalent weight, which significantly reduces computational complexity. In this paper, we present a linear-time algorithm for exactly computing the Potts model partition function on series-parallel graphs (SP-graphs).
\end{abstract}


\section{Introduction}
The purpose of this paper is to develop and describe an algorithm for calculating the partition function of the Potts model for series-parallel graphs with arbitrary weights. The Potts model is a powerful tool widely used in physics, mathematics, and graph theory. It allows us to model complex systems, including magnetic materials and crystal structures.

The multivariate Tutte polynomial, also known as the partition function of the Potts model (see~\cite{sokal}), can be defined for any finite graph and encodes key information about the structure of the graph. Let \( G = (V, E) \) be a finite undirected graph, where \( V \) is the set of vertices and \( E \) is the set of edges. The partition function is defined as follows:

\begin{equation}
\label{eq:pots}
Z_G(q, \mathbf{v}) = \sum_{A \subseteq E} q^{K(A)} \prod_{e \in A} v_e,
\end{equation}
where \( q \) and \( \mathbf{v} = \{v_e\}_{e \in E} \) are commuting variables, and \( K(A) \) is the number of connected components in subgraph \( (V, A) \). If we set \( v_e = -1 \) for all \( e \in E \), then \( Z_G \) takes the value~$1$ for each proper coloring of the graph and $0$ for each improper one, which allows the computation of the chromatic polynomial~\cite{sokal}.

\section{Preliminaries}

The standard definition of a series-parallel graph is: series–parallel graphs are graphs with two distinguished vertices called terminals, formed recursively from $K_2$ by two simple composition operations --- series and parallel composition.~\cite{sp}

In this paper, we provide an alternative definition of series-parallel graphs, equivalent to the standard one, but not starting from $K_2$; instead, it reduces an already existing SP-graph to $K_2$. But first, we introduce two definitions.

\begin{defin}
\label{def:def1}
Edges $e_1,\ldots,e_n$ are said to be \emph{parallel} if they connect the same pair of vertices $x$ and $y$.
\end{defin}

\begin{defin}
\label{def:def2}
Edges $e_1,\ldots,e_n$ form a \emph{series connection} between vertices $x$ and $y$ if they form a path in the graph between these vertices and all the internal vertices of the path have degree 2.
\end{defin}
In this paper,  by a parallel connection we mean several edges that connect the same pair of vertices (Definition 1), whereas a parallel composition is the graph operation that glues several SP-graphs in parallel at their terminals.

By a series connection, we mean a path of edges all of whose internal vertices have degree 2 (Definition 2), whereas a series composition is the graph operation that glues two SP-graphs in series at their terminals. 
Now we present an alternative definition of a series-parallel graph, different from the conventional one. 

A series-parallel graph is a graph that can be reduced to the graph $K_2$  by alternating the following operations:
\begin{enumerate}
    \item Replacing all series connections (paths) with a single edge;
    \item Replacing all parallel edges with a single edge.
\end{enumerate}
\FloatBarrier
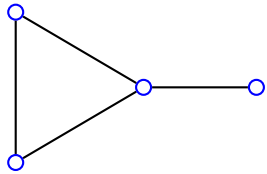
\begin{figure}
\centering
\begin{tikzpicture}
  \node[circle, fill=white, draw=blue, thick, inner sep=2pt] (A) at (-1,1) {};
  \node[circle, fill=white, draw=blue, thick, inner sep=2pt] (B) at (-1,-1) {};
  \node[circle, fill=white, draw=blue, thick, inner sep=2pt] (C) at (0.7,0) {};
  \node[circle, fill=white, draw=blue, thick, inner sep=2pt] (D) at (2.2,0) {};
  \draw[black, thick] (A) -- (B);
  \draw[black, thick] (A) -- (C);
  \draw[black, thick] (B) -- (C);
   \draw[black, thick] (C) -- (D);

\end{tikzpicture}
\caption{Example of a SP-graph}
\end{figure}
\FloatBarrier
\begin{utv} \label{p1}\cite{our,sokal}
If the graph $G$ contains $n$ parallel edges $e_1, \ldots, e_n$ connecting the same pair of vertices $x$ and $y$, they can be replaced by a single edge $e$ with weight:
\begin{equation}
\label{eq:parallel}
v_e = \prod_{i=1}^n \left(1 + v_{e_i}\right) - 1.
\end{equation}
\end{utv}
\begin{utv}\label{s1}\cite{our,sokal}
For series connections of edges $e_1, \ldots, e_n$, where there exists a chain of vertices $x_1, \ldots, x_{n+1}$ connected in series, one can use the following transformations:

1. Replace the sequence of edges $e_1, \ldots, e_n$ with a single edge $e = x_1x_{n+1}$ with weight:
\begin{equation}
\label{eq:series}
v_e = \frac{q \prod_{i=1}^n v_{e_i}}{\prod_{i=1}^n \left(v_{e_i} + q\right) - \prod_{i=1}^n v_{e_i}}.
\end{equation}

2. Multiply the partition function of the resulting graph by the prefactor:
\begin{equation}
\label{eq:prefactor}
\mathrm{pref} = \frac{\prod_{i=1}^n \left(v_{e_i} + q\right) - \prod_{i=1}^n v_{e_i}}{q}.
\end{equation}
\end{utv}
Recall (\cite{our}) that a graph $H$ is called a ``necklace'' graph if it has two vertices $x$ and $y$, called the ``source'' and  the ``sink'', respectively, such that all edges of $E(H)$ belong to some series connections between these vertices. Consequently, all the vertices of this graph, except $x$ and $y$, have degree 2 and also belong to series connections between the source and the sink. The different series connections between $x$ and $y$ are called the levels of the ``necklace'' graph and are numbered $0$ to $m$.

\begin{utv}[\cite{our}]
\label{utv12:box} 
Let $G$ contain a ``necklace'' subgraph $H$ with arbitrary weights $v_{ij}$, for $i=0,\ldots,m$ and $j=1,\ldots,k_i$, assigned to the edges of the $i$-th level of this ``necklace''. Suppose that all the vertices of $H$, except the source and the sink, are not incident to any other vertices of the graph $G$. Then the partition function $Z_G$ remains unchanged if the subgraph $H$ is replaced by a single edge $e$, connecting the source and the sink, with weight
\begin{equation}
\label{eg:z3}
v_e = \prod_{i=0}^m \left( \frac{\prod_{j=1}^{k_i} \left(v_{ij}+q\right) + (q-1)
\prod_{j=1}^{k_i} v_{ij}}{\prod_{j=1}^{k_i}
\left(v_{ij}+q\right)-\prod_{j=1}^{k_i} v_{ij}}\right) - 1,
\end{equation}
and the computed partition function is multiplied by the prefactor
$$
\prod_{i=0}^m \frac{\prod_{j=1}^{k_i} \left(v_{ij}+q\right)-\prod_{j=1}^{k_i} v_{ij}}{q}.
$$
\end{utv}
Before formulating the proposed algorithm, another idea for its operation was put forward and investigated.

It was that, in order to compute the partition function, one should identify the articulation points of an SP-graph and use them to decompose the original graph into ``necklace'' subgraphs. 






However, this algorithm has several problematic aspects. First, not every SP-graph contains articulation points, so this decomposition is not universally applicable.  However, there is a well-known linear - time algorithm~\cite{tarjan} for finding all articulation points in graphs. Moreover, finding all simple paths between ``source'' and ``sink'' is, in general case, would require a specialized adaptation for the SP-graph class. Each biconnected block with \( p \) alternative routes requires enumerating all \( p \) simple paths between the terminals of  blocks. If such blocks are composed in series, the total number of paths grows as a product of options. 

In a typical ``moderately parallel'' graph, this number grows faster than polynomially, but remains significantly below the worst-case exponential bound. However, if one simply applies the formula for computing the partition function~\ref{eq:pots}, the complexity becomes exponential.

These issues motivated the adoption of a different approach, the use of SP-decomposition trees.
\section{Algorithm for Finding the Partition Function for Series-Parallel Graphs}
An SPQR (``S'' stands for ``Series'', 
``P'' stands for ``Parallel'', 
``Q'' represents a single edge (a trivial component), 
``R'' stands for ``Rigid'' (a 3-connected component)  ) tree is a tree data structure used in computer science, and more specifically graph algorithms, to represent the triconnected components of a graphs.

 The SPQR-tree of a graph may be constructed in linear time and has several applications in dynamic graph algorithms and graph drawing.\cite{spq} There are well-known results (see \cite{spq1}) showing that SP-graphs cannot contain triconnected subgraphs; therefore, the definition of the SPQR-tree for them can be simplified by omitting the R-type nodes. An SPQ-tree, or decomposition tree, of an SP-graph is a binary tree whose internal nodes describe the type of edge connection in the graph (S for series connections and P for parallel connections) and whose leaves (Q) represent the single edges.

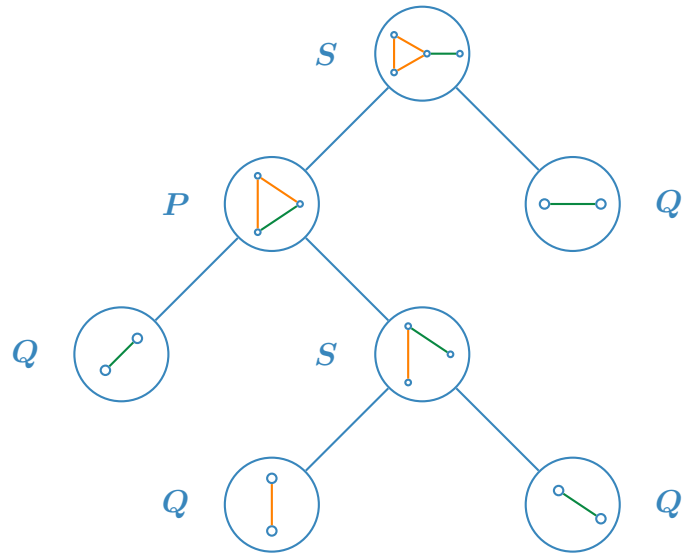
\begin{figure}[ht]
    \centering
    \begin{tikzpicture}[every node/.style={circle, draw=lightblue, minimum size=1.25 cm, thick}, 
                    edge label/.style={draw=lightblue, midway, font=\tiny, thick},
                    node distance=0.1 pt,
                    baseline=(current bounding box.base)
                    ]
        \node[label={[text=lightblue]left:$\boldsymbol{P}$}] (a) at (0,0) {};
        \node[label={[text=lightblue]right:$\boldsymbol{Q}$}] (b) at (4,0) {};
        \node[label={[text=lightblue]left:$\boldsymbol{S}$}] (c) at (2,2) {};
        \node[label={[text=lightblue]left:$\boldsymbol{S}$}] (d) at (2,-2) {};  
        \node[label={[text=lightblue]left:$\boldsymbol{Q}$}] (e) at (-2,-2) {};  
        \node[label={[text=lightblue]right:$\boldsymbol{Q}$}] (f) at (4,-4) {}; 
        \node[label={[text=lightblue]left:$\boldsymbol{Q}$}] (g) at (0,-4) {}; 

        \begin{scope}[shift={(a.center)}, scale=0.75]
          \node[minimum size=0.75 mm, inner sep=0pt] (a1) at (-0.25, 0.5) {};
          \node[minimum size=0.75 mm, inner sep=0pt] (a2) at (-0.25, -0.5) {};
          \node[minimum size=0.75 mm, inner sep=0pt] (a3) at (0.5, 0) {};
          \draw[draw=orange, thick] (a1) -- (a2);
          \draw[draw=lightgreen, thick] (a2) -- (a3);  
          \draw[draw=orange, thick] (a1) -- (a3);  
        \end{scope}

        \begin{scope}[shift={(b.center)}, scale=0.75]
          \node[minimum size=1.25mm, inner sep=0pt] (a1) at (-0.5, 0) {};
          \node[minimum size=1.25mm, inner sep=0pt] (a2) at (0.5, 0) {};
          \draw[draw=lightgreen, thick] (a1) -- (a2);
        \end{scope}

        \begin{scope}[shift={(c.center)}, scale=0.5]
          \node[minimum size=0.75 mm, inner sep=0pt] (a1) at (-0.75, 0.5) {};
          \node[minimum size=0.75 mm, inner sep=0pt] (a2) at (-0.75, -0.5) {};
          \node[minimum size=0.75 mm, inner sep=0pt] (a3) at (0.125, 0) {};
          \node[minimum size=0.75 mm, inner sep=0pt] (a4) at (1, 0) {};
          \draw[draw=orange, thick] (a1) -- (a2);
          \draw[draw=orange, thick] (a2) -- (a3);  
          \draw[draw=orange, thick] (a1) -- (a3);            
          \draw[draw=lightgreen, thick] (a3) -- (a4);
        \end{scope}

        \begin{scope}[shift={(d.center)}, scale=0.75]
          \node[minimum size=0.75 mm, inner sep=0pt] (a1) at (-0.25, 0.5) {};
          \node[minimum size=0.75 mm, inner sep=0pt] (a2) at (-0.25, -0.5) {};
          \node[minimum size=0.75 mm, inner sep=0pt] (a3) at (0.5, 0) {};
          \draw[draw=orange, thick] (a1) -- (a2); 
          \draw[draw=lightgreen, thick] (a1) -- (a3);  
        \end{scope}

        \begin{scope}[shift={(e.center)}, scale=0.85]
          \node[minimum size=1.25 mm, inner sep=0pt] (a2) at (-0.25, -0.25) {};
          \node[minimum size=1.25 mm, inner sep=0pt] (a3) at (0.25, 0.25) {};
          \draw[draw=lightgreen, thick] (a2) -- (a3);  
        \end{scope}

        \begin{scope}[shift={(f.center)}, scale=0.75]
          \node[minimum size=1.25mm, inner sep=0pt] (a1) at (-0.25, 0.25) {};
          \node[minimum size=1.25mm, inner sep=0pt] (a2) at (0.5, -0.25) {};
          \draw[draw=lightgreen, thick] (a1) -- (a2);
        \end{scope}

        \begin{scope}[shift={(g.center)}, scale=0.7]
          \node[minimum size=1.25mm, inner sep=0pt] (a1) at (0, 0.5) {};
          \node[minimum size=1.25mm, inner sep=0pt] (a2) at (0, -0.5) {};
          \draw[draw=orange, thick] (a1) -- (a2);
        \end{scope}
        \draw[draw=lightblue, thick] (a) -- (c);
        \draw[draw=lightblue, thick] (b) -- (c);
        \draw[draw=lightblue, thick] (a) -- (e);
        \draw[draw=lightblue, thick] (a) -- (d);
        \draw[draw=lightblue, thick] (f) -- (d);
        \draw[draw=lightblue, thick] (g) -- (d);
    \end{tikzpicture}
\caption{Example of a SP-tree}
\end{figure}

An SPQ-tree (hereafter referred to as an SP-tree) can be constructed via edge reduction rules that identify parallel and series connections until the graph is reduced to $K_2$. This process operates in O$(|V|+|E|$) time and uses O$(|V|+|E|$)  space and can be described as follows \cite{ufkapano}:
\begin{enumerate}[label=\arabic*.]
\item Find all parallel edges in the graph and reduce them using the formula \eqref{eq:parallel}. Thus, in the general case, transform a multigraph into a simple graph.
\item Find all vertices with degree $2$.
\item For each such vertices:
    \begin{enumerate}[label=\arabic{enumi}.\arabic*]
    \item Retrieve the adjacent vertices.
    \item Check if there is an edge between these adjacent vertices.
        \begin{itemize}
            \item If so, mark it as a parallel edge connection.
            \item Otherwise, it is a series connection.
        \end{itemize}
    \item Push the current vertex($c$), its adjacent vertices ($a$, $b$), and connection type onto the stack.
    \item For each adjacent vertex, check its degree.
        \begin{itemize}
            \item If its degree is $3$, add it to the list of vertices having degree $2$.
        \end{itemize}
    \item Remove the edges connecting the considered vertex to its adjacent vertices.
    \end{enumerate}
\item For the remaining graph, check the count of vertices with degree $1$.
    \begin{itemize}
        \item If the number of such vertices is not equal to $2$, then the input graph is not series-parallel.
        \item Otherwise, create a root node that represents the single edge.
    \end{itemize}
\item Build the SP-tree using stack.
\begin{itemize}
    \item For each recorded connection (``series'', ``parallel'' or single egde) and the vertices $c$, $a$, $b$ in the stack, the algorithm retrieves the corresponding node from the partial SP‑tree and replaces it with a new internal node representing that connection.
    \begin{itemize}
        \item In the case of a series connection, the SP-tree node corresponding vertices $a,b$ is replaced with an internal node representing a serial structure, where the original path from $a$ to $b$ is subdivided by introducing intermediate nodes, resulting in a path of two connected edges.
          \begin{figure}[ht]
            \centering
            \begin{subfigure}[h]{0.45\textwidth}
            \centering
            \begin{tikzpicture}[every node/.style={circle, draw}, 
                                edge label/.style={draw=none, midway, font=\tiny},
                                node distance=0.1 pt,
                                baseline=(current bounding box.base)]
              \node[label={left:$a$}] (a) at (0,0) {~};
              \node[label={right:$b$}] (b) at (3,0) {~};
              \node[label={[label distance=-0mm]above:$c$}] (c) at (1.5,1.5) {~};
              \draw (a) -- (c);
              \draw (b) -- (c);
            \end{tikzpicture}
            \captionsetup{skip=6mm}
            \caption{\centering{Example of an SP-graph having a series connection}}
            \end{subfigure}
            \hfill
            \begin{subfigure}[h]{0.45\textwidth}
            \centering
            \begin{tikzpicture}[every node/.style={circle, draw}, 
                                edge label/.style={draw=none, midway, font=\small},
                                baseline=(current bounding box.base)]
                \node[thick, minimum size=1.25 cm, draw=lightblue, label={[text=lightblue]left:$\boldsymbol{Q}$}] (a) at (0,0) {};
                \node[thick, minimum size=1.25 cm, draw=lightblue, label={[text=lightblue]right:$\boldsymbol{Q}$}] (b) at (4,0) {};
                \node[thick, minimum size=1.25 cm, draw=lightblue, label={[text=lightblue]left:$\boldsymbol{S}$}] (c) at (2,2) {};
        
                \begin{scope}[shift={(a.center)}, scale=0.75]
                  \node[minimum size=1.25mm, inner sep=0pt] (a1) at (-0.25, -0.25) {};
                  \node[minimum size=1.25mm, inner sep=0pt] (a2) at (0.25, 0.25) {};
                  \draw[draw=orange, thick] (a1) -- (a2);
                \end{scope}
        
                \begin{scope}[shift={(b.center)}, scale=0.75]
                  \node[minimum size=1.25mm, inner sep=0pt] (b1) at (0.25, -0.25) {};
                  \node[minimum size=1.25mm, inner sep=0pt] (b2) at (-0.25, 0.25) {};
                  \draw[draw=lightgreen, thick] (b1) -- (b2);
                \end{scope}
        
                \begin{scope}[shift={(c.center)}, scale=0.75]
                  \node[minimum size=1.25mm, inner sep=0pt] (c1) at (-0.5, -0.25) {};
                  \node[minimum size=1.25mm, inner sep=0pt] (c2) at (0.5, -0.25) {};
                  \node[minimum size=1.25mm, inner sep=0pt] (c3) at (0, 0.25) {};
                  \draw[draw=orange, thick] (c1) -- (c3);
                  \draw[draw=lightgreen, thick] (c2) -- (c3);
                \end{scope}
        
                \draw[draw=lightblue, thick] (a) -- (c);
                \draw[draw=lightblue, thick] (b) -- (c);
            \end{tikzpicture}
            \caption{\centering{SP-tree corresponding to the SP-graph from (a)}}
            \end{subfigure}
            \caption{Example for the case of series connection}
        \end{figure}
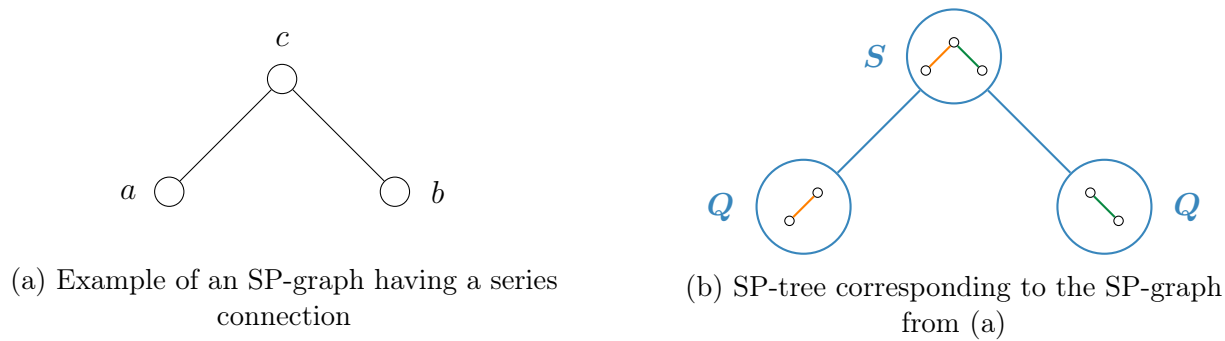
        \FloatBarrier
        \item  In the case of a parallel connection, the node becomes a parallel node with two ``children'': one representing a direct edge between the ``source'' and ``sink'', and the second representing a series connection of two vertices through the intermediate node.
        \begin{figure}[ht]
            \centering
            \begin{subfigure}[h]{0.45\textwidth}
            \centering
            \begin{tikzpicture}[every node/.style={circle, draw}, 
                                edge label/.style={draw=none, midway, font=\tiny},
                                node distance=0.1 pt,
                                baseline=(current bounding box.base)]
              \node[label={left:$a$}] (a) at (0,0) {~};
              \node[label={right:$b$}] (b) at (3,0) {~};
              \node[label={[label distance=-0mm]above:$c$}] (c) at (1.5,1.5) {~};
              \draw (a) -- (c);
              \draw (b) -- (c);
              \draw (a) -- (b);
            \end{tikzpicture}
            \captionsetup{skip=6mm}
            \caption{\centering{Example of an SP-graph having a parallel connection}}
            \end{subfigure}
            \hfill
            \begin{subfigure}[h]{0.45\textwidth}
            \centering
            \begin{tikzpicture}[every node/.style={circle, draw}, 
                                edge label/.style={draw=none, midway, font=\small},
                                baseline=(current bounding box.base)]
                \node[thick, minimum size=1.25 cm, draw=lightblue, label={[text=lightblue]left:$\boldsymbol{Q}$}] (a) at (0,0) {};
                \node[thick, minimum size=1.25 cm, draw=lightblue, label={[text=lightblue]right:$\boldsymbol{S}$}] (b) at (4,0) {};
                \node[thick, minimum size=1.25 cm, draw=lightblue, label={[text=lightblue]left:$\boldsymbol{P}$}] (c) at (2,2) {};
                \node[thick, minimum size=1.25 cm, draw=lightblue, label={[text=lightblue]left:$\boldsymbol{Q}$}] (d) at (2,-2) {};  
                \node[thick, minimum size=1.25 cm, draw=lightblue, label={[text=lightblue]right:$\boldsymbol{Q}$}] (e) at (6,-2) {};  
                \begin{scope}[shift={(a.center)}, scale=1.25]
                  \node[minimum size=1.25mm, inner sep=0pt] (a1) at (-0.25, 0) {};
                  \node[minimum size=1.25mm, inner sep=0pt] (a2) at (0.25, 0) {};                  
                  \draw[draw=lightgreen, thick] (a1) -- (a2);
                \end{scope}
        
                \begin{scope}[shift={(b.center)}, scale=0.75]
                  \node[minimum size=1.25mm, inner sep=0pt] (c1) at (-0.5, -0.25) {};
                  \node[minimum size=1.25mm, inner sep=0pt] (c2) at (0.5, -0.25) {};
                  \node[minimum size=1.25mm, inner sep=0pt] (c3) at (0, 0.25) {};
                  \draw[draw=orange, thick] (c1) -- (c3);
                  \draw[draw=lightgreen, thick] (c2) -- (c3);
                \end{scope}
        
                \begin{scope}[shift={(c.center)}, scale=0.75]
                  \node[minimum size=1.25mm, inner sep=0pt] (c1) at (-0.5, -0.25) {};
                  \node[minimum size=1.25mm, inner sep=0pt] (c2) at (0.5, -0.25) {};
                  \node[minimum size=1.25mm, inner sep=0pt] (c3) at (0, 0.25) {};
                  \draw[draw=orange, thick] (c1) -- (c3);
                  \draw[draw=orange, thick] (c2) -- (c3);
                  \draw[draw=lightgreen, thick] (c2) -- (c1);
                \end{scope}

                \begin{scope}[shift={(d.center)}, scale=0.75]
                  \node[minimum size=1.25mm, inner sep=0pt] (a1) at (-0.25, -0.25) {};
                  \node[minimum size=1.25mm, inner sep=0pt] (a2) at (0.25, 0.25) {};                  
                  \draw[draw=orange, thick] (a1) -- (a2);
                \end{scope}
                
                \begin{scope}[shift={(e.center)}, scale=0.75]
                  \node[minimum size=1.25mm, inner sep=0pt] (a1) at (0.25, -0.25) {};
                  \node[minimum size=1.25mm, inner sep=0pt] (a2) at (-0.25, 0.25) {};                  
                  \draw[draw=lightgreen, thick] (a1) -- (a2);
                \end{scope}   
                
                \draw[draw=lightblue, thick] (a) -- (c);
                \draw[draw=lightblue, thick] (b) -- (c);
                \draw[draw=lightblue, thick] (b) -- (d);
                \draw[draw=lightblue, thick] (b) -- (e);
            \end{tikzpicture}
            \caption{\centering{SP-tree corresponding to the SP-graph from (a)}}
            \end{subfigure}
            \caption{Example for the case of parallel connection}
        \end{figure}
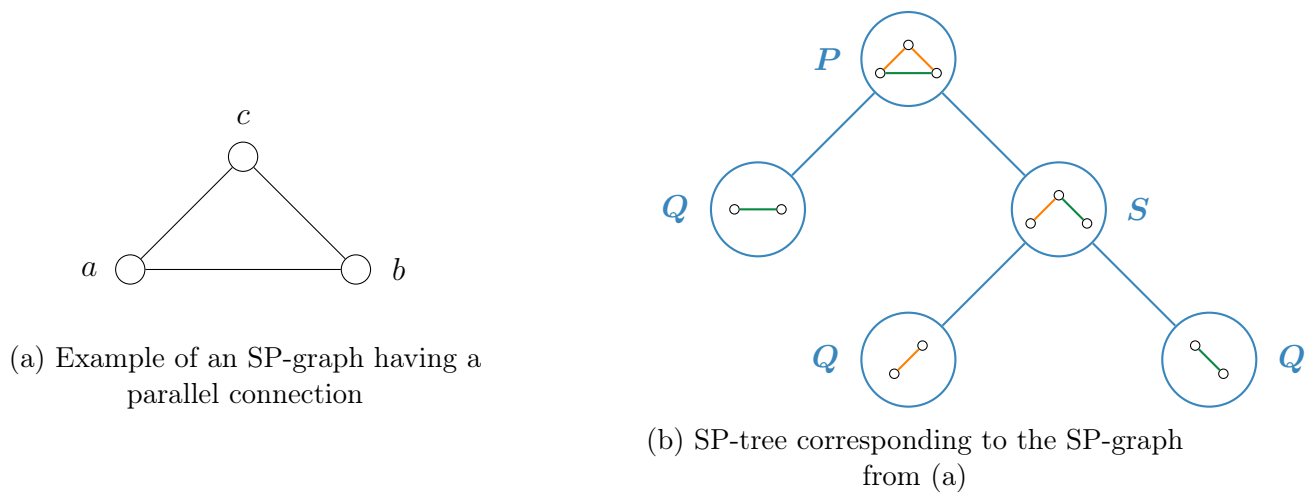 
    \end{itemize}
\end{itemize}
\end{enumerate}

This procedure can be summarized in the pseudocode below:
\par
\begin{algorithm}[H]
\caption{Build SP-tree}
\KwIn{SP-graph $G(V,E)$}
\KwOut{SP-tree rooted at node $root$} 

$\mathrm{G}_1(\mathrm{V}_1,\mathrm{E}_1) \gets$ \texttt{Copied graph $G$}\;
$\mathrm{degree2} \gets \{ v \mid v \in \mathrm{V}_1 : \deg(v) = 2 \}$\;
$\mathrm{stack} \gets \emptyset$\;
\While{$\mathrm{degree2} \neq \emptyset$}{
    $\mathrm{source} \gets$ \texttt{pop the element from } $\mathrm{degree2}$\;
    \If{$\mathrm{deg}(\mathrm{source}) \neq 2$}{
        \texttt{Skip cycle iteration}\;
    }
    $\mathrm{node}_1$ \texttt{ and } $\mathrm{node}_2 \gets$ \texttt{neighbors of } $\mathrm{source}$ \texttt{ vertex in } $\mathrm{G}_1$\;
    \If{$(\mathrm{node}_1, \mathrm{node}_2) \in \mathrm{E}_1$}{
        \texttt{Push onto } $\mathrm{stack}$ \texttt{ parallel composition type, } $\mathrm{source}$, $\mathrm{neighbor}_1$ \texttt{ and } $\mathrm{neighbor}_2$\;
    }
    \Else{
        $\mathrm{E}_1 \gets \mathrm{E}_1 \cup \{(\mathrm{node}_1, \mathrm{node}_2)\}$\;
        $\mathrm{V}_1 \gets \mathrm{V}_1 \cup \{ \mathrm{node}_1, \mathrm{node}_2 \}$\;
        \texttt{Push onto } $\mathrm{stack}$ \texttt{ series composition type, } $\mathrm{source}$, $\mathrm{neighbor}_1$ \texttt{ and } $\mathrm{neighbor}_2$\;
    }
    \If{$\deg(\mathrm{node}_1) = 3$}{
        $\mathrm{degree2} \gets \mathrm{degree2} \cup \{\mathrm{node}_1\}$\;
    }
    \If{$\deg(\mathrm{node}_2) = 3$}{
        $\mathrm{degree2} \gets \mathrm{degree2} \cup \{\mathrm{node}_2\}$\;
    }
    $\mathrm{E}_1 \gets \mathrm{E}_1 \setminus \{ (\mathrm{source}, \mathrm{node}_1), (\mathrm{source}, \mathrm{node}_2) \}$\;
}
$\mathrm{degree1} \gets \{ v \mid v \in \mathrm{V}_1 : \deg(v) = 1 \}$\;
\If{$|\mathrm{degree1}| = 2$}{
    $\mathrm{node}_1$ \texttt{ and } $\mathrm{node}_2 \gets$ \texttt{Vertices in } $\mathrm{degree1}$\;
    $\mathrm{root} \gets$ \texttt{Node of SP-tree with source and sink } $\mathrm{node}_1$ \texttt{ and } $\mathrm{node}_2$ \texttt{ respectively}\;
}
\Else{
    \texttt{raise Error}\;
}
\While{$\mathrm{stack} \neq \emptyset$}{
    $\mathrm{composition}, \mathrm{source}, \mathrm{node}_1, \mathrm{node}_2 \gets$ \texttt{Pop the last pushed element from } $\mathrm{stack}$\;
    \texttt{Construct a SP-tree fragment depending on the value of variable } $\mathrm{composition}$\;
}
\end{algorithm}
\FloatBarrier
To obtain a weight $w$ for $K_2$ that preserves the Potts partition function of the original graph, the SP‑tree is traversed in a depth‑first manner. At each internal node, the ``child'' weights are combined using the corresponding composition rules, as defined in \eqref{eq:series} for the series composition and \eqref{eq:parallel} for the parallel composition. 
The prefactor \eqref{eq:prefactor} introduced by the series composition rule is also taken into account. 
The depth-first traversal procedure used to compute the effective weight is described in the pseudocode below. Recursively, it visits the SP-tree nodes and applies the appropriate combination rule at each step, depending on the node type.
\begin{algorithm}
\caption{Compute equivalent weight on \(K_2\) from SP-tree}
\KwIn{SP-tree rooted at node $root$ and constructed from SP-graph $G(V,E)$}
\KwOut{Equivalent edge weight $w$}

\SetKwFunction{Traverse}{Traverse}
\SetKwProg{Fn}{Function}{:}{}

$pref \gets 1$\;
\BlankLine
\Fn{\Traverse{$node$}}{
$left \gets$ \Traverse{$node.left$}\;
$right \gets$ \Traverse{$node.right$}\;

\If{$node.type$ \texttt{is series connection type}}{
$w \gets \frac{q \cdot left \times right}{(left + q)(right + q) - left \times right}$\;
$pref \gets pref \times \frac{left \times right}{weight}$\;
}
\ElseIf{$node.type$ \texttt{is parallel connection type}}{
$w \gets (1 + left)  (1 + right) - 1$\;
}
\Else{
$w \gets$ \texttt{weight of edge } $(node.source, node.sink)$ \texttt{ in } $G$\;
}
\Return{$w$}
}
\end{algorithm}
\FloatBarrier
For the Potts model in the graph $K_2$, the partition function is given by the simple explicit formula:
\begin{equation}
\label{eq:K2}
Z = \bigl(q^2 + q \times w\bigr) \times pref
\end{equation}

\section { Correctness of the algorithm}
To prove the correctness of the algorithm, we show that it transforms any SP-graph into a valid decomposition tree and correctly computes the Potts model partition function. The critical part is ensuring the decomposition tree construction is valid, since the subsequent stages (computing equivalent weights and the final sum) follow standard SP-graph reduction rules and known results. We rely on several properties of series–parallel (SP) graphs:

\begin{utv} \label {utv1} In any series-parallel graph (SP-graph), constructed from a single edge using only series and parallel connection, there exists at least one vertex of degree at most~2. In other words, every SP-graph has a vertex of degree 1 or 2.
\end{utv}
\begin{proof}

SP-graphs are defined recursively. The base case is the graph $K_2$, consisting of two vertices $s$ and $t$ connected by a single edge. If $X$ and $Y$ are SP-graphs with designated terminals $(s_X, t_X)$ and $(s_Y, t_Y)$, then we can create a new SP-graph in two ways:

\begin{itemize}
    \item \textbf{Series composition:} Identify $t_X$ with $s_Y$. The new terminals are $s_X$ and $t_Y$.
    \item \textbf{Parallel composition:} Identify $s_X$ with $s_Y$ and $t_X$ with $t_Y$. The new terminals are the merged pairs.
\end{itemize}

By recursively applying these operations to copies of $K_2$, we can construct any SP-graph.

Let us prove it by induction.

We prove the proposition by induction on the number of composition operations used in the construction.

\textbf{Base case:} In $K_2$, both vertices have degree 1. The claim holds.

\textbf{Inductive step:} Suppose any SP-graph constructed in $n$ steps has a vertex of degree $\leq 2$. Let $G$ be an SP-graph constructed in $n{+}1$ steps via series or parallel composition of $G_1$ and $G_2$.

\textbf{Case 1: Series Composition.} Let $G = Series_{composition}(G_1, G_2)$, where $t_1$ from $G_1$ is identified with $s_2$ from $G_2$. Denote this merged vertex by $v$. Then
\[
\deg_G(v) = \deg_{G_1}(t_1) + \deg_{G_2}(s_2).
\]
By the inductive hypothesis, $G_1$ has a vertex $x$ with $\deg_{G_1}(x) \leq 2$, and similarly for $G_2$. If $x \neq t_1$ and $y \neq s_2$, then both $x$ and $y$ remain valid low-degree vertices in $G$. If $x = t_1$ and $y = s_2$, then $v$ has
\[
\deg_G(v) \leq 4,
\]
but at least one of the graphs must have another low-degree vertex unless all other vertices have degree $> 2$, which contradicts the inductive assumption. Hence, $G$ contains a vertex of degree $\leq 2$.

\textbf{Case 2: Parallel Composition.} Let $G = Parallel_{composition}(G_1, G_2)$ where $s_1 \equiv s_2 = u$ and $t_1 \equiv t_2 = w$. Then
\[
\deg_G(u) = \deg_{G_1}(s_1) + \deg_{G_2}(s_2), \quad
\deg_G(w) = \deg_{G_1}(t_1) + \deg_{G_2}(t_2).
\]
By the inductive hypothesis, both $G_1$ and $G_2$ contain a vertex of degree $\leq 2$. If these are internal vertices, they remain in $G$. If they are terminal, say $x = s_1$ and $y = s_2$, then $u$ has
\[
\deg_G(u) = \deg_{G_1}(s_1) + \deg_{G_2}(s_2) \leq 4.
\]
Even if $u$ or $w$ has degree up to 4, one of the other vertices still satisfies the degree bound, preserving the claim.

Alternatively, SP-graphs can be reduced to $K_2$ using inverse operations—removal of parallel edges and suppression of degree-2 vertices. Since SP-graphs contain no $K_4$ minors and are never 3-connected, if all vertices had degree at least 3, reduction to $K_2$ would be impossible. Hence, a vertex of degree $\leq 2$ must always exist.

In all possible constructions of an SP-graph via series and parallel composition, the resulting graph contains a vertex of degree at most 2. This completes the inductive proof.
\end{proof}
\begin{utv} \label{prop5}
Let \(G=(V,E)\) be any simple series–parallel graph constructed from \(K_2\) by repeated series and parallel compositions. Then
\[
|E|\le 2|V|-3.
\]
\end{utv}

\begin{proof}
For planar graphs and bipartite graphs, the relationships between the number of edges and vertices are well known~\cite{distel}. Although the article~\cite{sp} contains all the necessary information  for deriving an estimate 
$|E|\le 2|V|-3$, but the statement itself and its proof are not given explicitly, so here we will provide a proof of this statement.
We proceed by induction on the number of edges \(m=|E|\).

\textbf{Base case.} For the graph \(K_2\), we have \(|V|=2\) and \(|E|=1\).  Hence
\[
|E| = 1 \le 2\cdot 2 -3 = 1,
\]
so the claim holds.

\textbf{Inductive step.} Assume the statement holds for all simple SP-graphs with fewer than \(m\) edges.  Let \(G\) be a simple SP-graph with \(|E|=m\).  Then \(G\) is obtained from two smaller SP-graphs \(G_1=(V_1,E_1)\) and \(G_2=(V_2,E_2)\) by either a series composition or a parallel composition.

\medskip

\emph{Case 1: Series composition.}  In a series composition, one terminal of \(G_1\) is identified with one terminal of \(G_2\).  Thus
\[
|V| = |V_1| + |V_2| -1,
\quad
|E| = |E_1| + |E_2|.
\]
By the inductive hypothesis,
\[
|E_1|\le 2|V_1|-3,
\quad
|E_2|\le 2|V_2|-3.
\]
Therefore
\[
|E| = |E_1|+|E_2|
\le (2|V_1|-3)+(2|V_2|-3)
= 2|V_1|+2|V_2|-6.
\]
Since \(2|V|=2|V_1|+2|V_2|-2\), it follows that
\[
|E|\le 2|V|-4 < 2|V|-3,
\]
and so the desired inequality holds in this case.

\medskip

\emph{Case 2: Parallel composition.}  In a parallel composition, both terminals of \(G_1\) are identified with the corresponding terminals of \(G_2\).  Hence
\[
|V| = |V_1| + |V_2| -2,
\quad
|E| = |E_1| + |E_2|.
\]
Moreover, since the resulting graph must remain simple, at least one of the graphs, say \(G_1\), cannot already have a direct edge between its two terminals (otherwise identification would create a double edge).  Consequently \(G_1\) has strictly fewer than its maximal allowed edges, and by a refined inductive hypothesis
\[
|E_1|\le 2|V_1|-4,
\quad
|E_2|\le 2|V_2|-3.
\]
Adding these gives
\[
|E| = |E_1|+|E_2|
\le (2|V_1|-4)+(2|V_2|-3)
= 2|V_1|+2|V_2|-7
= 2(|V_1|+|V_2|-2)-3
=2|V|-3,
\]
as required.

\medskip

By induction, the inequality \(|E|\le2|V|-3\) holds for all simple SP-graphs.
\end{proof}

\begin{utv}
In an SP-graph \( G = (V, E) \), if there exists exactly one vertex of degree 2, then there must also exist a vertex of degree 3.
\end{utv}

\begin{proof}
Assume the contrary: suppose that there exists an SP-graph \( G = (V, E) \) such that it has only one vertex of degree 2 (denote it by \( v_0 \)), and all other vertices have degree at least 4.

According to handshaking lemma the sum of degrees of all vertices of a graph is equal to twice the number of its edges:
\begin{equation}
    \sum_{v \in V} \deg(v) = 2|E|. \tag{3.2}
\end{equation}

Now use the previously established inequality~\ref{prop5}:
\[
|E| \leq 2|V| - 3 \quad \Rightarrow \quad 2|E| \leq 4|V| - 6.
\]

Therefore, the total degree satisfies:
\[
\sum_{v \in V} \deg(v) = \deg(v_0) + \sum_{v \in V \setminus \{v_0\}} \deg(v) \leq 4|V| - 6.
\]

Since all vertices except \( v_0 \) have degree at least 4, we also have:
\[
\sum_{v \in V} \deg(v) \geq \deg(v_0) + 4(|V| - 1) = 2 + 4(|V| - 1) = 4|V| - 2.
\]

Combining both estimates, we obtain:
\[
4|V| - 2 \leq \sum_{v \in V} \deg(v) \leq 4|V| - 6,
\]
which is a contradiction.

Hence, the assumption is false, and the theorem is proved.
\end{proof}

\begin{utv}
For decomposition trees built from the same SP-graph, the partition function of the Potts model computed on them is identical.
\end{utv}
As stated in~\cite{valdes} SP-trees constructed for the same SP-graph are equivalent up to the order of series and parallel compositions, then by the Church–Rosser theorem~\cite{church} any such tree can be reduced to a single canonical form. Since, according to propositions~\ref{p1},\ref{s1}, the partition function of the Potts model remains invariant under the reduction of parallel and series edges (with the appropriate weight substituted for the new edge), the values of the partition function computed for any two SP-trees built from the same SP-graph coincide.
The performance data obtained during this work were validated and confirmed by additional experiments that empirically assessed the complexity of the algorithm. To simplify these experiments without sacrificing significant accuracy, the following assumptions were made:
\begin{itemize}
  \item The parameter \(q\), denoting the number of possible states of a vertex in the graph \(G(V,E)\), is treated as a constant in the range \([1,40]\) rather than as a variable.
  \item The weight of each edge $e$ in $G(V, E)$ is randomly assigned a value from the range $[10^{-5}, 5 \times 10^{-2}]$.
  \item All statistical outliers in the collected data are excluded from the final dataset used to plot the graph.
\end{itemize}
\begin{figure}[H]
    \centering
    \includegraphics[width=1\linewidth]{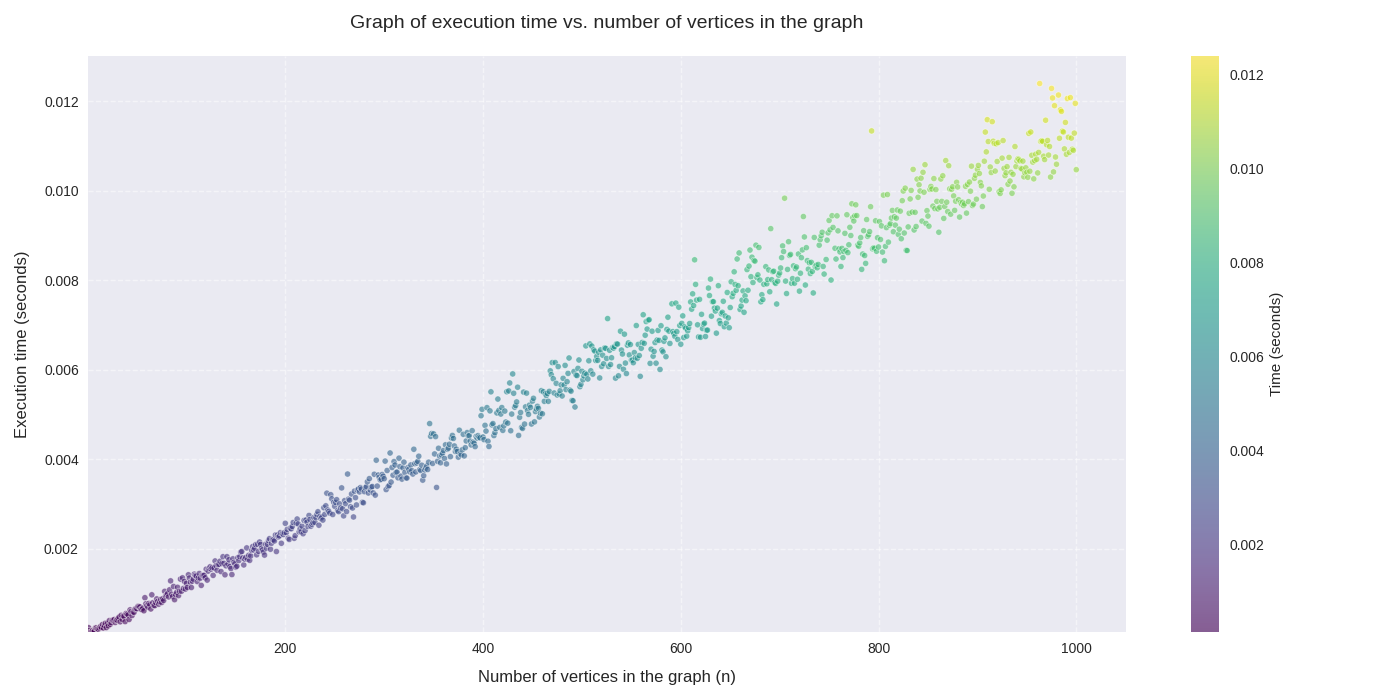}
    \caption{Execution Time of the Proposed Algorithm as a Function of Graph Size}
    \label{fig:linear}
\end{figure}

\noindent For a clear visual comparison, we also conducted experiments computing the Potts model partition function by its direct definition (see Fig.~\ref{fig:exponential}). As the resulting plots demonstrate, the proposed algorithm for SP‐graphs is far more efficient than the classical approach. For example, at \(n=20\) vertices, the traditional method requires about 140\,s to produce an answer, whereas the new algorithm takes less than 0.005\,s.
\begin{figure}[H]
    \centering
    \includegraphics[width=0.8\linewidth]{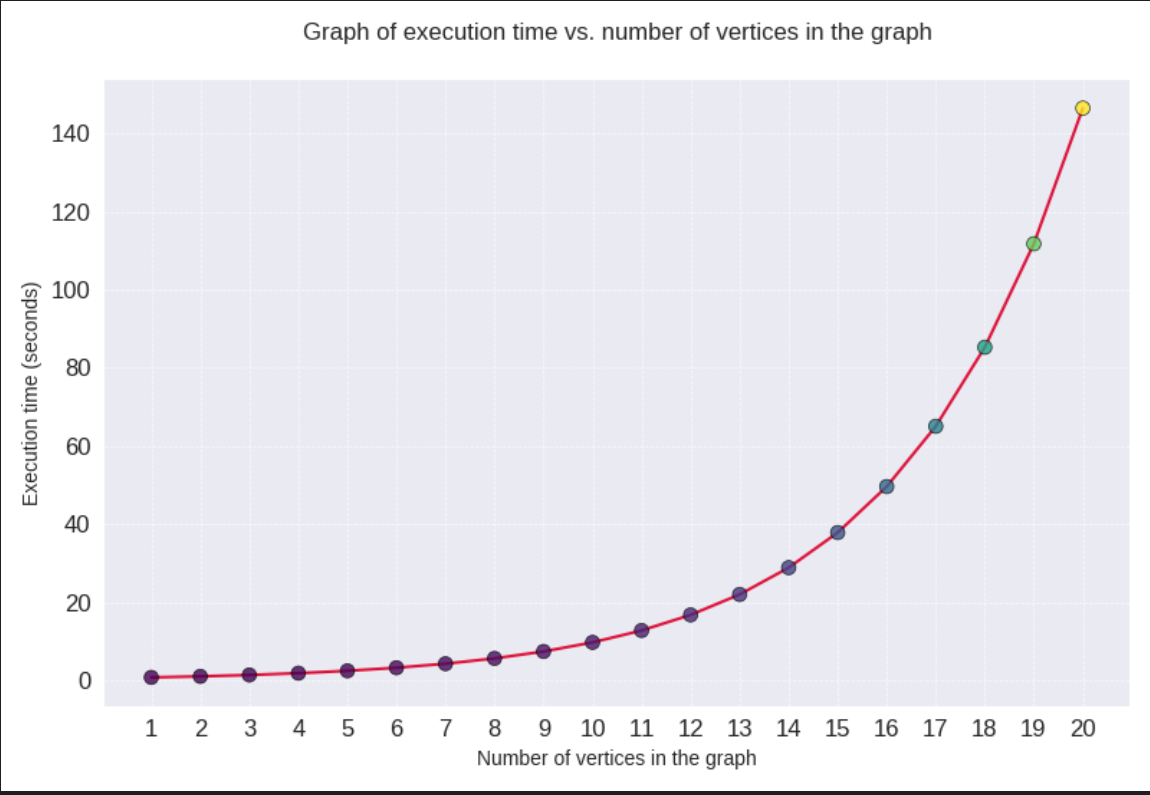}
    \caption{Inefficient scaling of the direct Potts partition function computation with graph size}
    \label{fig:exponential}
\end{figure}
\section{Conclusion}
We have developed an efficient linear-time algorithm for computing the partition function of the Potts model (i.e., the multivariate Tutte polynomial) for series-parallel graphs with arbitrary edge weights. The algorithm extends Sokal's series- and parallel-reduction identities \cite{sokal} and presents SP-graphs as SPQ-trees. By iteratively applying these reductions, the algorithm transforms the original graph into an equivalent two-vertex (single-edge) graph, from which the partition function can be obtained directly. As a result of these simplifications, the computational complexity of evaluating the Potts model partition function on series-parallel graphs using SPQ-trees is significantly reduced to linear complexity.


\end{document}